\documentclass[11pt]{article}

\usepackage{amsthm}
\usepackage{latexsym}
\usepackage{dsfont}
\usepackage{bbm}
\usepackage{amssymb}
\usepackage{amsmath}
\usepackage{graphicx}

\numberwithin{equation}{section}

\theoremstyle{plain}
\newtheorem{Theorem}{Theorem}[section]
\newtheorem{Lemma}[Theorem]{Lemma}

\theoremstyle{remark}
\newtheorem{Rem}[Theorem]{Remark}
\theoremstyle{definition}

\DeclareMathOperator{\N}{\mathbb{N}}

\DeclareMathOperator{\R}{\mathbb{R}}

\DeclareMathOperator{\Prob}{\mathbb{P}}

\DeclareMathOperator{\E}{\mathbb{E}}

\DeclareMathOperator{\1}{\mathbbm{1}}

\newcommand{\mn}{\mathbb{N}}

\newcommand{\lin}{\underset{n\to\infty}{\lim}}
\newcommand{\lix}{\underset{x\to\infty}{\lim}}

\title{Functional limit theorems for Galton-Watson processes with very active immigration}
\author{Alexander Iksanov\footnote{Faculty of Computer Science and Cybernetics, Taras Shevchenko National University of Kyiv, 01601 Kyiv, Ukraine \ \ e-mail:
iksan@univ.kiev.ua} \ and Zakhar Kabluchko\footnote{Institut f\"{u}r Mathematische Statistik, Westf\"{a}lische Wilhelms-Universit\"{a}t M\"{u}nster, 48149 M\"{u}nster, Germany \ \ e-mail: zakhar.kabluchko@uni-muenster.de}}

\begin{document}

\thispagestyle{empty}
\maketitle

\begin{abstract}
We prove weak convergence on the Skorokhod space of Galton-Watson processes with immigration, properly normalized, under the assumption that the tail of the immigration distribution has a logarithmic decay. The limits are extremal shot noise processes.
By considering marginal distributions, we recover the results of Pakes~[\textit{Adv.\ Appl.\ Probab.},  11(1979), 31-–62].
%One-dimensional versions of our results were obtained earlier by Pakes~[\textit{Adv.\ Appl.\ Probab.},  11(1979), 31-–62].
%in \cite{Pakes:1979}.

\vspace*{1mm}

\vspace*{1mm}
\hphantom{2000 Mathematics Subject Classification: }Secondary: 60J80
\end{abstract}

\section{Introduction and main result} \label{sec:Intro_and_main_results}

In this paper we are concerned with the Galton-Watson processes
(GW processes, in short) with immigration. Below we outline the
setting and refer to the classical treatises
\cite{Asmussen+Hering:1983, Athreya+Ney:1972} for more details on
the GW processes with and without immigration.

Let $(X_{i,k})_{i\in\mn, k\in\mn_0}$ and $(J_k)_{k\in\mn_0}$, where $\mn_0:=\mn\cup\{0\}$, be mutually independent families of independent and
identically distributed (i.i.d.)\ random objects, where each
$X_{i,k}:=(X_{i,k}(n))_{n\in\mn_0}$ is a GW
process with $X_{i,k}(0)=1$ and $\E X_{i,k}(1)=\mu\in (0,\infty)$,
and each $J_k$ is a nonnegative integer-valued random variable with
$\Prob\{J_k=0\}<1$. The random
sequence $Y:=(Y_n)_{n\in\mn_0}$ defined by
$$
Y_n:=\sum_{k=0}^n\sum_{i=1}^{J_k}X_{i,k}(n-k),\quad n\in\mn_0
$$
is called a {\it Galton-Watson process with immigration}. The random variable $J_k$ represents the number of immigrants which arrived at time $k$, while the GW process $(X_{i,k}(n))_{n\in\mn_0}$ represents the number of descendants of the $i$th immigrant which arrived at time $k$, for $1\leq i \leq J_k$.

Let $J$ denote a random variable with the same law as the $J_k$'s.
It is known that the asymptotic behavior of $Y_n$ depends heavily upon the finiteness of the logarithmic moment $\E \log^+ J$, where $\log^+ x = \max(\log x,0)$. In the supercritical case $\mu>1$, the a.s.\ limit $\lim_{n\to\infty} Y_n/\mu^n$ exists and is finite a.s. provided that $\E \log^+ J<\infty$, whereas  $\lim_{n\to\infty} Y_n/c^n=\infty$ a.s.\ for every $c>0$ if $\E \log^+ J = \infty$; see~\cite{Seneta:1970a}, \cite{Seneta:1970b}.
%there is no normalizing sequence $c_n$ such that $Y_n/c_n$ converges a.s.\ to some a.s.\ finite random variable which does not vanish identically.
In the subcritical case $\mu<1$, $Y_n$ converges in distribution to a non-degenerate random variable provided that $\E \log^+ J<\infty$, whereas $Y_n$ diverges to $+\infty$ in probability if $\E \log^+ J = \infty$; see~\cite{Heathcote:1966}. These and more refined results can be found in~\cite{Asmussen+Hering:1983}, Theorems 6.1 and 6.4.

Let $D:=D[0,\infty)$ denote the Skorokhod space of
right-continuous functions defined on $[0,\infty)$ with finite
limits from the left at positive points.
We intend to prove functional limit theorems for the process $\log^+(Y_{[n\cdot]})$ in $D$ as $n\to\infty$ under the assumption
\begin{equation}\label{stand}
\Prob\{\log J>x\}\sim x^{-\alpha}\ell(x),\quad x\to\infty
\end{equation}
for some $\alpha\in (0,1]$ and some $\ell$ slowly varying at $\infty$, which justifies the term ``very active immigration''.
%Under condition~\eqref{stand}, the extremal order statistics of the sequence $(J_k)_{k\in\N_0}$ play an essential role in the large $n$ limit.

To state our results we need to introduce certain Poisson random measures which appear as limits for the extremal order statistics of the sequence $(\log^+ J_k)_{k\in\N_0}$. Denote by $M_p$ the set of point measures $\nu$ on $[0,\infty)\times (0,\infty]$ which satisfy
\begin{equation}\label{1}
\nu([0,T]\times [\delta,\infty])<\infty
\end{equation}
for all $T>0$ and all $\delta>0$.  For positive $a$ and $b$, let
$$
N^{(a, b)}:=\sum_{k=1}^{\infty}
\varepsilon_{(t_k^{(a,b)},\,j_k^{(a,b)})}
$$
be a Poisson
random measure on $[0,\infty)\times (0,\infty]$ with mean measure
$\mathbb{LEB}\times \mu_{a, b}$, where $\varepsilon_{(t,\,x)}$
is the probability measure concentrated at $(t,x)\in [0,\infty)\times (0,\infty]$, $\mathbb{LEB}$ is the Lebesgue
measure on $[0,\infty)$, and $\mu_{a,b}$ is a measure on
$(0,\infty]$ defined by
$$
\mu_{a, b}\big((x,\infty]\big)=ax^{-b}, \ \ x>0.
$$
Throughout the paper we
use $\Rightarrow$ to denote weak convergence on the Skorokhod
space $D$ equipped with the $J_1$-topology (see \cite{Billingsley:1968, Lindvall:1973} for the necessary background) and on $M_p$ endowed with the vague topology.

Theorem \ref{main1} treats the situation in which the behaviour in
mean of the GW processes $X_{i,k}$ affects the limit behavior of
$Y$ (except in the less interesting case $\mu=1$), whereas in the
situation of Theorem \ref{main2} any traces of the $X_{i,k}$
disappear in the limit. We stipulate hereafter that the supremum
over the empty set is equal to zero.
\begin{Theorem}\label{main1} Assume that for some $c>0$,
\begin{equation}\label{2}
\lix x\Prob\{\log J>x\}=c.
\end{equation}
Then, as $n\to\infty$,
\begin{equation}\label{317}
\frac {\log^+ (Y_{[n\cdot]})}{n}\quad \Rightarrow \quad \underset{t_k^{(c,1)}\leq
\cdot}{\sup}\big(j_k^{(c,1)} + (\cdot-t_k^{(c,1)})\log\mu\big).
\end{equation}
\end{Theorem}
\begin{Rem}
Realizations of the limit processes, which are called extremal shot noise processes~\cite{Dombry:2012}, are shown on Figure~\ref{figure:extremal_shot_noise}. In the critical case $\mu=1$, the limit is the well-known extremal process; see~\cite{Resnick:2008}, Sections 4.3--4.4.
We shall see in the course of the proof that in the supercritical case $\mu>1$, we also have
\begin{equation}\label{31}
\frac {\log^+ (\mu^{-[n\cdot]}Y_{[n\cdot]})} {n}\quad\Rightarrow\quad \underset{t_k^{(c,1)}\leq
\cdot}{\sup}\big(j_k^{(c,1)} - t_k^{(c,1)} \log\mu\big),
\end{equation}
in which case the marginal distributions of the limit process have support equal to $[0,\infty)$.
\end{Rem}

\begin{figure}[t]
\begin{center}
\includegraphics[width=0.32\textwidth]{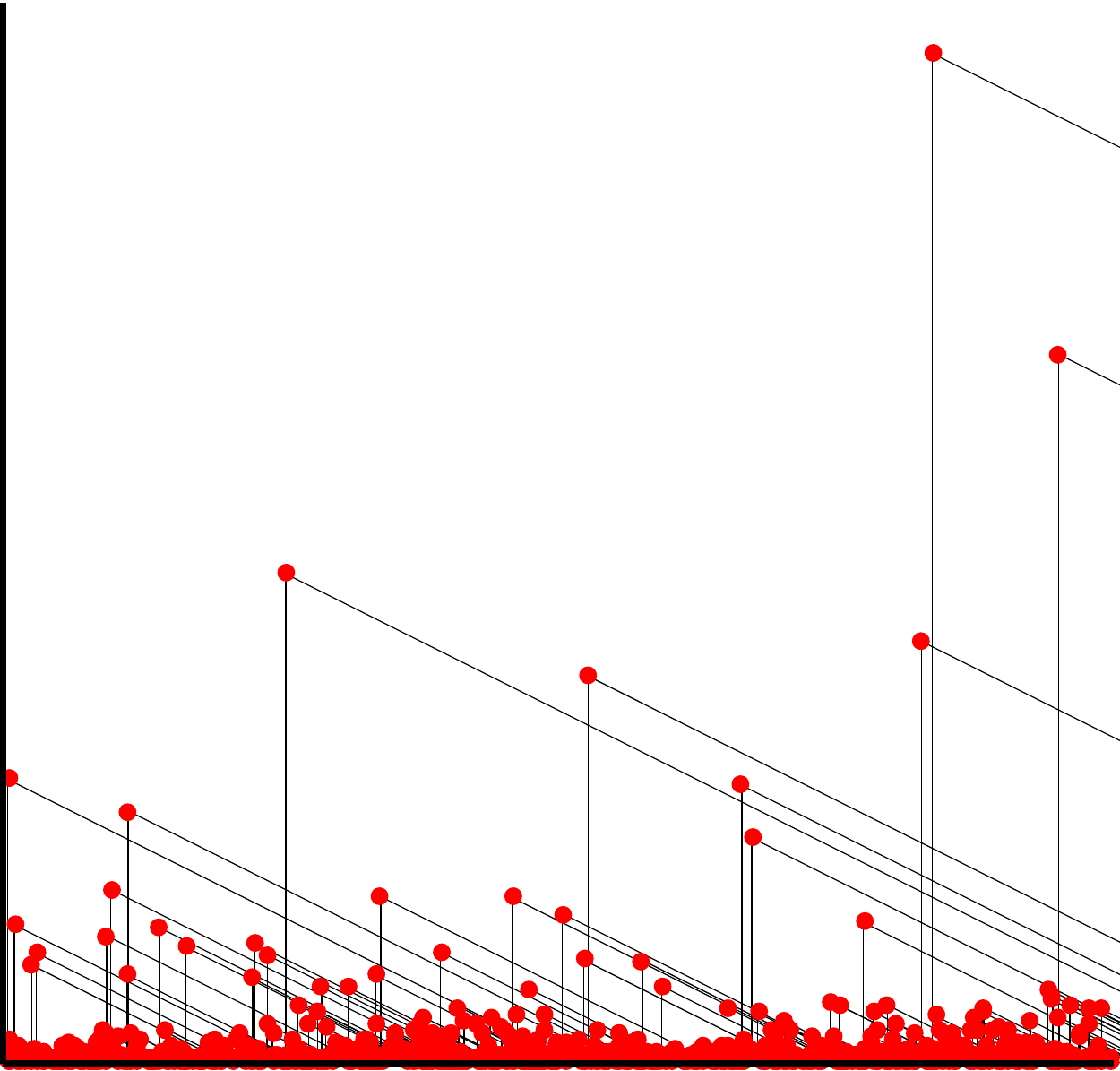}
\includegraphics[width=0.32\textwidth]{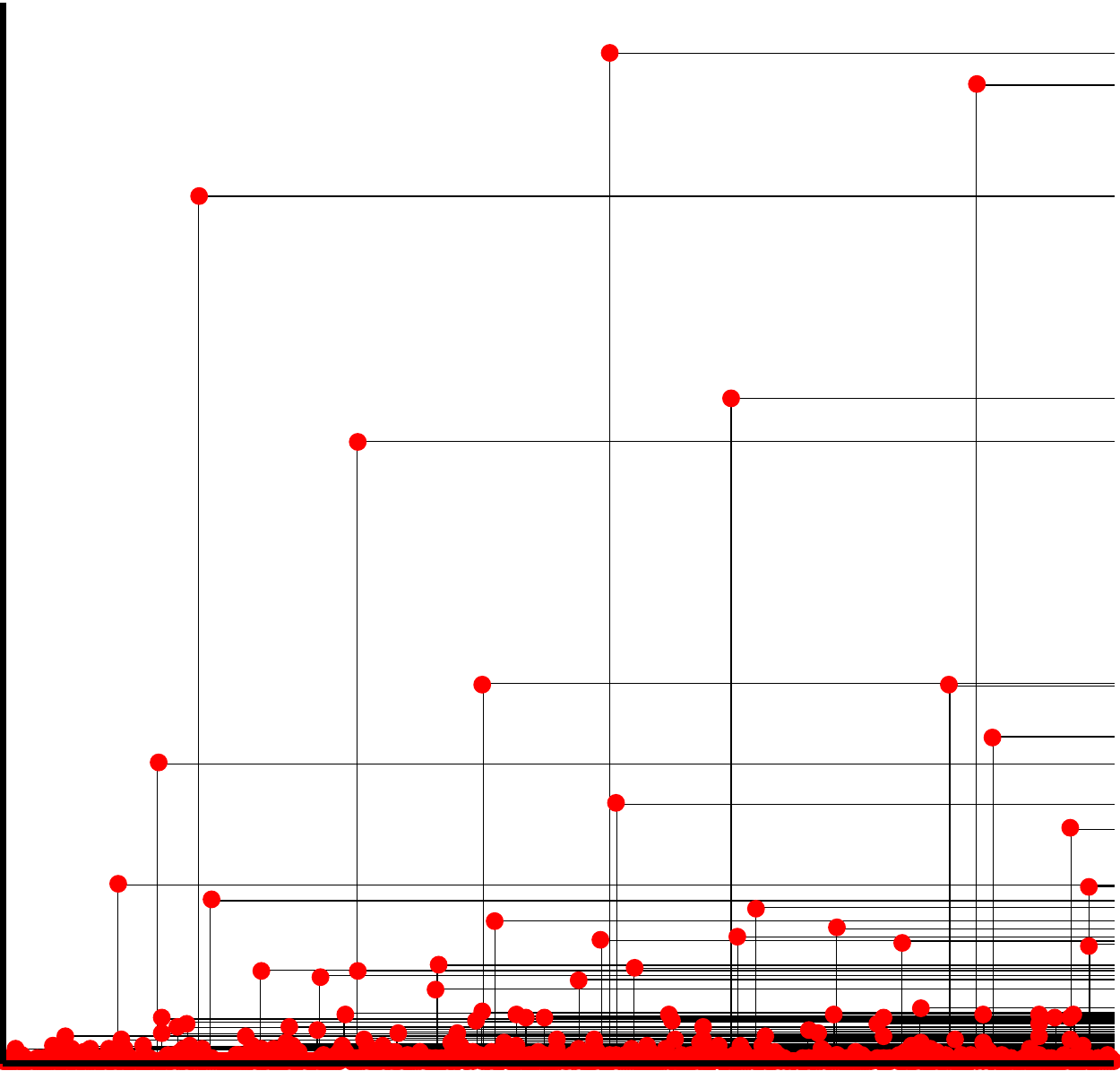}
\includegraphics[width=0.32\textwidth]{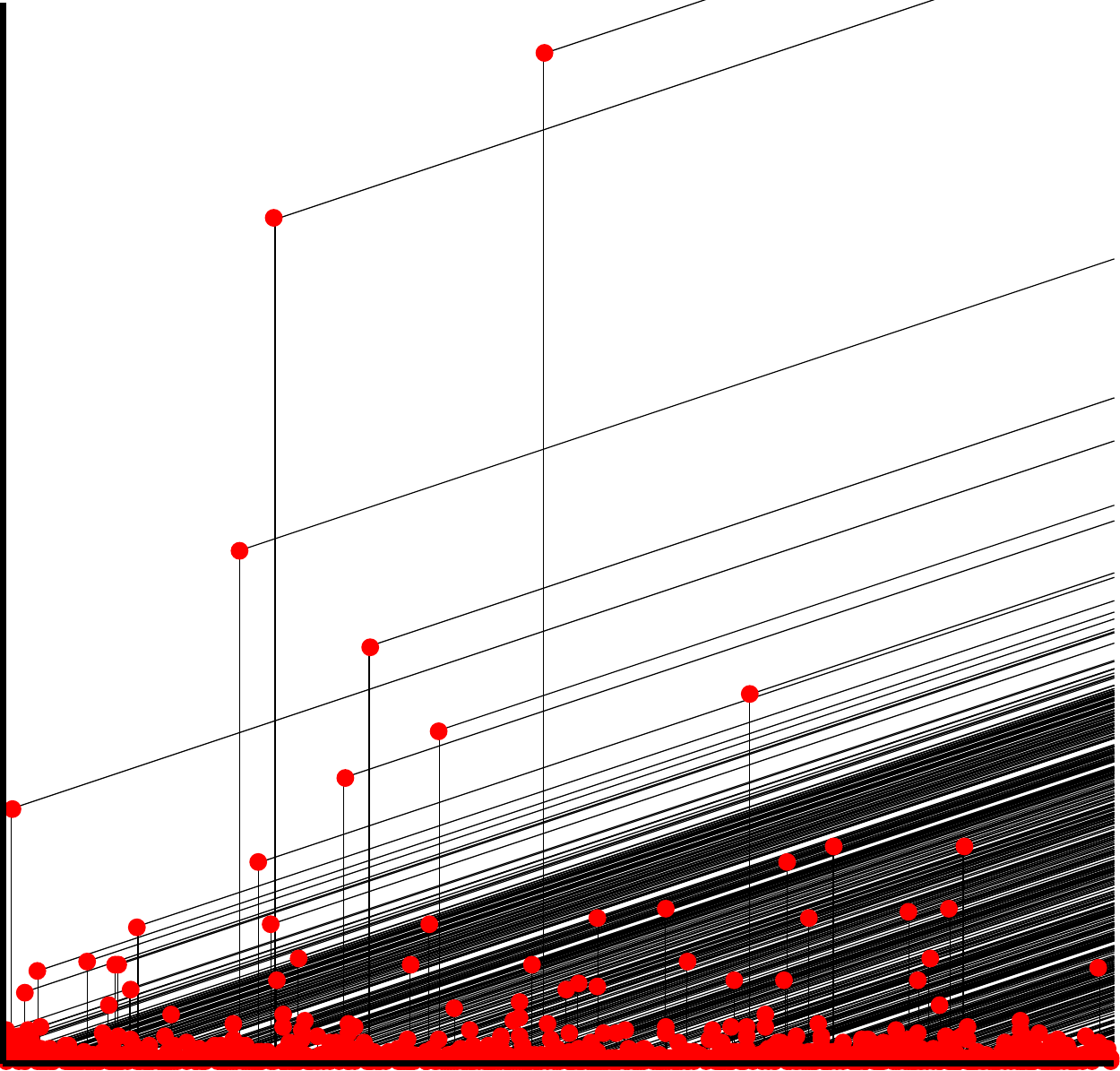}
\end{center}
\caption
{\small The limit processes appearing in Theorem~\ref{main1}. Left: subcritical case $\mu<1$.  Middle: critical case $\mu=1$. Right: supercritical case $\mu>1$. }
\label{figure:extremal_shot_noise}
\end{figure}

\begin{Theorem}\label{main2}
Suppose that \eqref{stand} holds. If $\alpha$ in \eqref{stand} equals $1$, suppose additionally that
$\lix \ell(x)=\infty$. Let $(b_n)_{n\in\N}$ be a sequence of positive numbers satisfying
$$
\lin n\Prob\{\log J>b_n\}=1.
$$
Then, as $n\to\infty$,
\begin{equation}\label{35}
\frac {\log^+ (Y_{[n\cdot]})} {b_n}\quad\Rightarrow\quad \underset{t_k^{(1,\,\alpha)}\leq
\cdot}{\sup}\,j_k^{(1,\alpha)}.
\end{equation}
\end{Theorem}

\begin{Rem}
Let us derive closed formulae for the marginal distributions of the limit processes. We claim that, with $r,s>0$ and $u\geq 0$,
\begin{eqnarray}\label{38}
\Prob\left\{\underset{t_k^{(r,1)}\leq
u}{\sup}\left(j_k^{(r,1)} + s (t_k^{(r,1)}-u)\right)\leq x\right\}
&=&
\Prob\left\{\underset{t_k^{(r,1)}\leq
u}{\sup}\left(j_k^{(r,1)}-st_k^{(r,1)}\right)\leq x\right\}\notag\\&=&\left(\frac {x}
{x+su}\right)^{r/s}
\end{eqnarray}
for all $x\geq 0$.
We only provide details for the second probability. Since
$$
N:= N^{(r,1)}\big((t,y): 0\leq t\leq u,
y-st>x\big)
$$
is a Poisson random variable, we have $\Prob\big\{N=0\big\}= e^{-\E N}$ and it remains to note that
\begin{eqnarray*}
\E N=\int_0^u\int_{[0,\infty)}\1_{\{y-st>x\}}\mu_{r,\,1}({\rm d}y){\rm
d}t=r\int_0^u(x+st)^{-1}{\rm
d}t=\frac rs \log\frac {x+su}{x}.
\end{eqnarray*}
Similarly, for the marginals of the extremal process appearing in~\eqref{35} we obtain, with $a,b>0$ and $u\geq 0$,
\begin{equation*}%\label{43}
\Prob\left\{\underset{t_k^{(a,\,b)}\leq
u}{\sup}\,j_k^{(a,\,b)}\leq
x\right\}=\Prob\left\{N^{(a, b)}\left((t,y): 0\leq t\leq u,
y>x\right)=0\right\}= e^{-uax^{-b}}
\end{equation*}
for all $x\geq 0$. Armed with these observations we conclude that
relations \eqref{317} and \eqref{31} include the
results obtained by Pakes \cite{Pakes:1979} concerning weak convergence
of the one-dimensional distributions (Theorem 2, Theorem 6 and
Theorem 12 for the subcritical $\mu<1$, supercritical $\mu>1$ and
critical
$\mu=1$ cases, respectively). Similarly, the one-dimensional version of our relation \eqref{35}
is equivalent to the limit relations of Theorem 3 (case $\mu<1$),
Theorem 7 (case $\mu>1$) and Theorem 12 (case $\mu=1$) of
\cite{Pakes:1979}. To be more precise, Pakes states in Theorems 3
and 7 that
\begin{equation}\label{pak}
\lin \Prob\bigg\{{\log(1+Y_n)\over a_n}\leq
x\bigg\}=\exp(-(|\log\mu|)^{-1}x^{-\alpha}),\quad x\geq 0
\end{equation}
whenever
\begin{equation}\label{pak1}
1-\E(1-e^{-x})^J\sim x^{-\alpha}\ell(x)\quad x\to\infty
\end{equation}
with $a_n$ defined by $1-\E(1-e^{-a_n})^J \sim (|\log \mu|n)^{-1}$
as $n\to\infty$. Formula \eqref{pak} is misleading because it
contains $|\log \mu|$ thereby suggesting that the contribution of the
$X_{i,k}$ persists in the limit. However, the relation
\begin{equation}\label{ours}
\lin \Prob\bigg\{{\log(1+Y_n)\over b_n}\leq x\bigg\}=\exp(- x^{-\alpha}),\quad x\geq 0
\end{equation}
which is a consequence of \eqref{35} shows this is not the case.
Having observed that \eqref{pak1} is equivalent to \eqref{stand}
we infer $b_n\sim (|\log \mu|)^{-1/\alpha}a_n$ which implies that
\eqref{pak} and \eqref{ours} are actually equivalent.

Finally,  we note that unlike us, Pakes \cite{Pakes:1979} imposed a regular
variation assumption on the tail of $X_{1,1}(1)$ for the critical
case and used the Seneta-Heyde norming rather than
$\mu^{-n}$ in the supercritical case in \eqref{31}.
\end{Rem}

\begin{Rem}
One may expect that, under \eqref{stand}, $Y_n$ is well approximated by $Z_n:=\E(Y_n|(J_k)_{k\in\N_0})=\sum_{k=0}^n \mu^{n-k}J_k$ for large $n$. Although this turns out to be true, it is worth stressing that both the behavior in mean and the survival probability (especially in the subcritical case) of the underlying GW processes affect the asymptotics of $Y_n$.
%These influences are quantitatively different. While the former is sometimes seen in the limit, this is never the case for the latter.
The sequence $(Z_n)_{n\in\N_0}$ is a rather particular case\footnote{Just take $A_n=\mu$ and $B_n=J_n$, $n\in\N_0$.} of the much studied Markov chain $(X_n)_{n\in\N_0}$ defined by
$$X_0:=B_0\quad\text{and}\quad X_n=A_nX_{n-1}+B_n,\quad n\in\N,$$ where $(A_n, B_n)$ are i.i.d.\ $\R^2$-valued random vectors independent of $B_0$.
Functional limit theorems for $\log^+ (X_{[n\cdot]})$ were obtained in
\cite{Buraczewski+Iksanov:2015} under the assumption that
\eqref{stand} holds with $B_1$ replacing $J$ and that $\lin
(A_1\cdot\ldots\cdot A_n) =0$ a.s.\ (this corresponds to the
subcritical GW processes).
\end{Rem}

\section{Preparatory results}

 We start with a lemma that might have been known. Recall that $(X_{1,1}(n))_{n\in\N_0}$ is a GW process with $X_{1,1}(0)=1$ and mean $\mu \in (0,\infty)$.
\begin{Lemma}\label{aux1}
If $\mu\leq 1$, then for all $\delta>0$,
\begin{equation}\label{subexp}
\lin e^{\delta n}\mu^{-n}\Prob\{X_{1,1}(n)\geq 1\}=\infty.
\end{equation}
\end{Lemma}
\begin{proof}
Put $p_n:=\Prob\{X_{1,1}(n)\geq 1\}$, $n\in\mn$. If
$\Prob\{X_{1,1}(1)=1\}=1$, then $\mu=1$ and $p_n=1$, and
\eqref{subexp} holds trivially. To deal with the remaining cases
$\mu=1$ and $\Prob\{X_{1,1}=1\}<1$ or $\mu<1$ in which $\lin
p_n=0$ we set $f_n(s):=\E s^{X_{1,1}(n)}$, $n\in\mn$, $s\in
[0,1]$. Then
$$\mu=\underset{s\to 1-}{\lim}{1-f_1(s)\over 1-s}=\lin
{1-f_1(f_n(0))\over 1-f_n(0)}=\lin {p_{n+1}\over p_n}$$ having
utilized $1-p_n=\Prob\{X_{1,1}(n)=0\}=f_n(0)\to 1-$ as
$n\to\infty$ for the second equality. This implies $\lin
{\mu^{-(n+k)}p_{n+k}\over \mu^{-n} p_n}=1$ for each $k\in\mn$ and
thereupon \eqref{subexp}.
\end{proof}

%We shall decompose the process $Y$ into the contribution coming from times with extremely high immigration, and the contribution of all other times.
Let $(c_n)_{n\in\mn}$ be a sequence of positive numbers satisfying $\lin n^{-1}c_n=\infty$ or $c_n=n$ for all $n\in\mn$. For each $0<\gamma<1$, set
\begin{equation}\label{eq:def_Y_proc_truncated}
Y_{[n\cdot]}^{(\leq \gamma)}:=\sum_{k=0}^{[n\cdot]}\1_{\{J_k\leq e^{\gamma c_n}\}}\sum_{i=1}^{J_k}X_{i,k}([n\cdot]-k).
\end{equation}
The next lemma shows that $Y_{[n\cdot]}^{\leq \gamma}$, the contribution coming from times in which immigration is not extremely active, is negligible as $n\to\infty$ and $\gamma\to 0+$. No assumptions on the tail of $J$ are imposed in this lemma.
\begin{Lemma}\label{aux3}
Fix $T>0$ and $\gamma>0$.  For every $\delta>0$,
$$
\sum_{n\geq 1}\Prob\bigg\{\underset{0\leq t\leq T}{\sup}\,\frac 1 {c_n} \log^+\bigg(m([nt])Y_{[nt]}^{(\leq \gamma)}\bigg)>\gamma+\delta\bigg\} < \infty,
$$
%\begin{equation*}
%\underset{\gamma\to 0}{\lim}\,\underset{n\to\infty}{\lim\sup}\,\underset{0\leq t\leq T}{\sup}\,(c_n)^{-1}\log^+\bigg(m([nt])Y_{[nt]}^{(\leq %\gamma)}\bigg)=0
%\quad\text{a.s.},
%\end{equation*}
where $m(n):=\mu^{-n}\wedge 1$, $n\in\N_0$ when $c_n=n$ and $m(n)=1$, $n\in\N_0$ when $\lin n^{-1}c_n=\infty$.
\end{Lemma}
\begin{proof}
Suppose first that $c_n=n$. For $r\in\N_0$, set $d(r):=(\mu^{-r}\wedge 1)\sum_{k=0}^r\mu^k$ and note that $d(r)=r+1$ if $\mu=1$ and $d(r)\leq (1-\mu\wedge\mu^{-1})^{-1}<\infty$ if $\mu\neq 1$. For all $\delta>0$,
\begin{eqnarray*}
&&\sum_{n\geq 1}\Prob\bigg\{\underset{0\leq t\leq T}{\sup}\,\frac 1n \log^+\bigg(\big(\mu^{-[nt]}\wedge 1\big)Y_{[nt]}^{(\leq \gamma)}\bigg)>\gamma+\delta\bigg\}\\&\leq &\sum_{n\geq 1}\sum_{r=0}^{[nT]}\Prob\bigg\{\big(\mu^{-r}\wedge 1\big)\sum_{k=0}^r\1_{\{J_k\leq e^{\gamma n}\}}\sum_{i=1}^{J_k}X_{i,k}(r-k)>e^{(\gamma+\delta)n}\bigg\}\\&\leq& \sum_{n\geq 1}\sum_{r=0}^{[nT]}\Prob\bigg\{\big(\mu^{-r}\wedge 1\big)\sum_{k=0}^r\sum_{i=1}^{[e^{\gamma n}]}X_{i,k}(r-k)>e^{(\gamma+\delta)n}\bigg\}\\&\leq & \sum_{n\geq 1} e^{-\delta n}\sum_{r=0}^{[nT]}
d(r)<\infty.
\end{eqnarray*}
We have used Boole's inequality for the second line and Markov's inequality in combination with $\E X_{i,k}(r-k)=\mu^{r-k}$ for the fourth line.

When $\lin n^{-1}c_n=\infty$, the proof is similar: for all $\delta>0$,
\begin{eqnarray*}
&&\sum_{n\geq 1}\Prob\bigg\{\underset{0\leq t\leq T}{\sup}\,\frac 1 {c_n}\log^+\bigg(Y_{[nt]}^{(\leq \gamma)}\bigg)>\gamma+\delta\bigg\}\\&\leq &\sum_{n\geq 1}\sum_{r=0}^{[nT]}\Prob\bigg\{\sum_{k=0}^r\1_{\{J_k\leq e^{\gamma c_n}\}}\sum_{i=1}^{J_k}X_{i,k}(r-k)>e^{(\gamma+\delta)c_n}\bigg\}\\&\leq& \sum_{n\geq 1}\sum_{r=0}^{[nT]}\Prob\bigg\{\sum_{k=0}^r\sum_{i=1}^{[e^{\gamma c_n}]}X_{i,k}(r-k)>e^{(\gamma+\delta)c_n}\bigg\}\\&\leq & \sum_{n\geq 1} e^{-\delta c_n}([nT]+1)\sum_{k=0}^{[nT]}\mu^k<\infty,
\end{eqnarray*}
because $([nT]+1)\sum_{k=0}^{[nT]}\mu^k$ grows at most exponentially whereas $e^{-\delta c_n}$ decreases superexponentially in view of $\lin n^{-1}c_n=\infty$.
\end{proof}

In the next lemma, which is needed to prove Theorem~\ref{main1}, we identify the response functions of the limit extremal shot noise process. Roughly speaking, this lemma states that a GW process with finite mean $\mu$ starting at time $0$ with approximately $e^{an+o(n)}$ particles has at time $nt$ approximately $\mu^{nt} e^{an +o(n)}$ particles. However, there is one exception: if the process is subcritical, then the population dies out approximately at time $n a/ |\log \mu|$, and the number of particles after this time is $0$.
\begin{Lemma}\label{aux2}
Let $(A_n)_{n\in\mn}\subset \N$ be a sequence satisfying $\lin n^{-1} \log A_n = a$ for some $a>0$.
Then, for every $T>0$ and every sequence $(k_n)_{n\in\N}\subset \N_0$,
\begin{equation}\label{limit1}
\lin
\sup_{0\leq t\leq T} \left|\frac 1n \log^+\bigg(\sum_{i=1}^{A_n}X_{i,k_n}([nt])\bigg) - (a+t \log\mu)^+\right| = 0
\quad \text{{\rm a.s.}}
\end{equation}
\end{Lemma}
\begin{proof}
According to the Borel-Cantelli lemma in combination with
$$\lin \underset{t\in [0,T]}{\sup}\, |[nt]/n-t|=0$$ it suffices to check that for all $\varepsilon\in (0,a)$,
$$
\sum_{n\geq 1}\Prob\bigg\{\underset{0\leq t\leq T}{\sup}\,\bigg|\log^+\bigg(\sum_{i=1}^{A_n}X_{i,0}([nt])\bigg)-(an+[nt]\log\mu)^+\bigg|>\varepsilon n\bigg\}<\infty.
$$
We write $X^\ast_{i,0}(l):=\mu^{-l}X_{i,0}(l)$, so that $\E X^\ast_{i,0}(l)=1$. By Boole's inequality the last probability is bounded from above by $I_1(n)+I_2(n)$ with
\begin{align*}
I_1(n) &:=
\sum_{r=0}^{[nT]}\Prob\bigg\{\sum_{i=1}^{A_n}X_{i,0}(r)> e^{(an+r\log\mu)^++\varepsilon n}\bigg\},\\
I_2(n) &:= \sum_{r=0}^{r_n}
\Prob\bigg\{\sum_{i=1}^{A_n}X^\ast_{i,0}(r)<e^{(a-\varepsilon)n}\bigg\},
\end{align*}
where  $r_n:=[nT]\wedge [nT_0]$ and
$$
T_0:=
\begin{cases}
(a-\varepsilon)/|\log\mu|,& \text{ if } \mu<1,\\
+\infty, &\text{ if } \mu\geq 1.
\end{cases}
$$

To prove that $\sum_{n\geq 1} I_1(n)$ is finite, note that $|\log A_n-an|\leq \frac \varepsilon 2n$ for sufficiently large $n$.
Using Markov's inequality yields
$$
\sum_{n\geq 1}I_1(n)\leq \sum_{n\geq 1}\sum_{r=0}^{[nT]} \frac{e^{\log A_n+r\log\mu}}{e^{ (an+r\log\mu)^++\varepsilon n}}
\leq
C+ \sum_{n\geq 1}([nT]+1)e^{-\frac 12 \varepsilon n}<\infty.$$

In the following, we prove that $\sum_{n\geq 1} I_2(n)$ is finite. Let $p_n:=\Prob\{X_{1,0}(n)\geq 1\}$, $n\in\N$, be the probability that a GW process starting with a single particle at time $0$  does not die out at time $n$.  We fix $u>0$ and use Markov's inequality in combination with the
fact that $(e^{-u X^\ast_{1,0}(l)})_{l\in\N_0}$ is a
submartingale w.r.t.\ the natural filtration to infer that for large enough $n$ and all $r\leq r_n$,
\begin{align}
\log\Prob \left\{\sum_{i=1}^{A_n} X^\ast_{i,0}(r)<e^{(a-\varepsilon)n}\right\}
&\leq
ue^{(a-\varepsilon)n} + A_n\log \E e^{-uX^\ast_{1,0}(r_n)} \label{import} \\
&\leq
ue^{(a-\varepsilon)n}+A_n\log(1-p_{r_n}+p_{r_n}e^{-\mu^{-r_n}u}) \notag\\
&\leq
ue^{(a-\varepsilon)n}-A_n(1-e^{-\mu^{-r_n}u})p_{r_n} \notag\\
&\leq ue^{(a-\varepsilon)n}-e^{(a-\frac 12 \varepsilon)n}(1-e^{-\mu^{-r_n}u})p_{r_n}.\label{import2}
\end{align}
Further, we consider the three cases separately.

\vspace*{2mm}
\noindent {\sc Supercritical case $\mu>1$}. It is well known (see, for instance, Theorem 5.1 on p.~83 together
with Corollary 5.3 on p.~85 in \cite{Asmussen+Hering:1983}) that
there exists a function $L$ slowly varying at $\infty$ with $\liminf_{x\to\infty} L(x)>0$
such that, as $n\to\infty$, $X^\ast_{1,0}(n)/L(\mu^n)$ converges
a.s.\ to a random variable $W$, say, which is positive with
positive probability\footnote{The sequence $(\mu^n L(\mu^n))$ is
known as the Seneta-Heyde norming.}. In particular, by the dominated convergence,
$$
c := \lin \log \E e^{-X^\ast_{1,0}(n)/L(\mu^n)} = \log \E e^{-W}   < 0.
$$
With
$u=1/L(\mu^{[nT]})$, inequality \eqref{import} takes the form
$$
\log\Prob\left\{\sum_{i=1}^{A_n} X^\ast_{i,0}(r)<e^{(a-\varepsilon)n}\right\} < \frac{e^{(a-\varepsilon) n}}{L(\mu^{[nT]})} + (c+o(1))A_n.
$$
The right-hand side goes to $-\infty$ as $n\to\infty$ exponentially fast  because $A_n > e^{(a-\frac 12 \varepsilon)n}$ for large $n$ and $L$ is slowly varying, thereby proving $\sum_{n\geq 1}I_2(n)<\infty$.

\vspace*{2mm}
\noindent {\sc Critical case $\mu=1$}. With $u>0$ fixed, inequality \eqref{import2} takes the form
$$
\log\Prob\left\{\sum_{i=1}^{A_n} X^\ast_{i,0}(r)<e^{(a-\varepsilon)n}\right\}
\leq ue^{(a-\varepsilon)n} - e^{(a-\frac 12 \varepsilon) n}(1-e^{-u})p_{[nT]}.
$$
In view of \eqref{subexp} this goes to $-\infty$ as $n\to\infty$ exponentially fast, whence $\sum_{n\geq 1}I_2(n)<\infty$.

\vspace*{2mm}
\noindent {\sc Subcritical case $\mu<1$}. With $u=e^{-(a-\varepsilon)n}$, expression \eqref{import2} takes the form
$$
1-(1-e^{-\mu^{-[nT_0]} e^{-(a-\varepsilon)n}})p_{[nT_0]}e^{(a-\frac 12 \varepsilon)n}
\leq
1-(1-e^{-\mu}) p_{[nT_0]}\mu^{-{[nT_0]}}e^{\frac 12 \varepsilon n},
$$
because $\mu<\mu^{-{[nT_0]}}e^{-(a-\varepsilon)n}\leq 1$.  In view of \eqref{subexp} this goes to $-\infty$ as $n\to\infty$ exponentially fast, thus proving that $\sum_{n\geq 1}I_2(n)<\infty$ in this case, too.
\end{proof}
In the next lemma, we consider a GW process starting at time $0$ with $e^{(a + o(1)) c_n}$ particles, where, as before, $(c_n)_{n\in\N}$ is a sequence of positive numbers satisfying $\lim_{n\to\infty} n^{-1}c_n =\infty$. At time $nt$, the number of particles in such a process is approximately $\mu^{nt} e^{(a+o(1))c_n} = e^{(a+o(1))c_n}$, so that we do not see any changes on the logarithmic scale. The subcritical case plays no special role here, because the process is very unlikely to die out on the time scale $n$.

\begin{Lemma}\label{aux2a}
Let $(A_n)_{n\in\N}\subset \N$ be a sequence satisfying $\lin c_n^{-1} \log A_n = a$ for some $a>0$.
Then, for every $T>0$ and every sequence $(k_n)_{n\in\N}\subset \N_0$,
\begin{equation}\label{limit2}
\lin \sup_{0\leq t\leq T} \left| \frac 1 {c_n}\log^+\bigg(\sum_{i=1}^{A_n}X_{i,k_n}([nt])\bigg) - a\right| =0
\quad \text{{\rm a.s.}}
\end{equation}
\end{Lemma}
\begin{proof}
The proof  is similar to that of Lemma \ref{aux2}. Therefore we only give an outline. It suffices to prove that  for all $\varepsilon\in (0,a)$,
$$
\sum_{n\geq 1}\Prob\bigg\{\underset{0\leq t\leq T}{\sup}\,\bigg|\log^+\bigg(\sum_{i=1}^{A_n}X_{i,0}([nt])\bigg)-ac_n\bigg|>\varepsilon c_n\bigg\}<\infty.
$$
The last probability is bounded from above by $J_1(n)+J_2(n)$, where
\begin{align*}
J_1(n) &:= \sum_{r=0}^{[nT]}\Prob\bigg\{\sum_{i=1}^{A_n}X_{i,0}(r)>\exp((a+\varepsilon)c_n)\bigg\}\\
J_2(n) &:= \sum_{r=0}^{[nT]}
\Prob\bigg\{\sum_{i=1}^{A_n}X^\ast_{i,0}(r)<\mu^{-r}\exp((a-\varepsilon)c_n)\bigg\}.
\end{align*}
For sufficiently large $n$, we have $|\log A_n-ac_n|\leq \frac \varepsilon 2 c_n$.
We use Markov's inequality to obtain
$$
\sum_{n\geq 1}J_1(n)
\leq
\sum_{n\geq 1}  A_n e^{-(a+\varepsilon)c_n}  \sum_{r=0}^{[nT]} \mu^r
\leq C+
\sum_{n\geq 1}e^{-\frac12 \varepsilon c_n}\sum_{r=0}^{[nT]}\mu^r<\infty,
$$ where the finiteness follows from the fact that
$e^{-\frac 12 \varepsilon c_n}$ decreases superexponentially in view of $\lin n^{-1}c_n=\infty$. While analyzing $J_2(n)$ we only treat the subcritical case. A counterpart of \eqref{import2} reads
$$
\log\Prob\bigg\{\sum_{i=1}^{A_n}X^\ast_{i,0}(r)<\mu^{-r}e^{(a-\varepsilon)c_n}\bigg\}\notag\leq u\mu^{-[nT]}e^{(a-\varepsilon)c_n}-e^{(a-\frac 12 \varepsilon)c_n} (1-e^{-\mu^{-[nT]}u})p_{[nT]}
$$
for large enough $n$, $r\leq [nT]$ and any $u>0$. On setting $u=\mu^{[nT]}$ the right-hand side takes the form
$$ e^{(a-\varepsilon)c_n}(1-(1-e^{-1}) e^{\frac 12 \varepsilon c_n} p_{[nT]}).$$ As $n\to\infty$, this goes to $-\infty$ superexponentially fast by \eqref{subexp}. The proof of Lemma \ref{aux2a} is complete.
\end{proof}

\begin{Lemma}\label{lem:log_+_subadd}
For all $x,y\geq 0$, we have
\begin{equation}\label{ineq}
\log^+x\leq \log^+(x+y)\leq \log^+ x+\log^+y+2\log 2.
\end{equation}
\end{Lemma}
\begin{proof}
While the left-hand inequality follows by monotonicity, the
right-hand inequality is a consequence of $$\log^+ x\leq \log
(1+x)\leq \log^+ x+\log 2,\quad x\geq 0$$ and the subadditivity of
$x\mapsto \log (1+x)$, namely,
\begin{eqnarray*}
\log^+(x+y)\leq \log (1+x+y)\leq \log (1+x)+\log
(1+y)\leq \log^+ x+\log^+y+2\log 2.
\end{eqnarray*}
\end{proof}

\section{Proofs of the main results}

\begin{proof}[Proof of Theorem \ref{main1}]
It is a standard fact of the extreme-value theory that
condition \eqref{2} entails the point processes convergence
\begin{equation*}
N_n := \sum_{k\geq 0}\1_{\{J_{k}\neq 0\}}\varepsilon_{(\frac kn,\,
\frac 1 n \log J_{k})}\quad \Rightarrow\quad
N^{(c,1)}
\end{equation*}
weakly on $M_p$, as $n\to\infty$; see, for instance, Corollary 4.19 (ii) on p.~210 in \cite{Resnick:2008}.

\vspace*{2mm}
\noindent
\textsc{Step 1: Passing to a.s.\ convergence.}
By the Skorokhod representation theorem there are versions $\widehat N_n$ and $\widehat N^{(c,1)}$ of $N_n$ and $N^{(c,1)}$ (defined on some new probability space) which converge
a.s.  That is, with probability $1$,
\begin{equation}\label{7}
\widehat{N}_n:=\sum_{k\geq 0} \1_{\{y_k^{(n)}\neq -\infty\}} \varepsilon_{(\frac kn,\,
y_k^{(n)})}\quad \to\quad
\widehat{N}^{(c,1)}=\sum_{m\geq 0}\varepsilon_{(\tau_m,\,y_m)},
\end{equation}
vaguely on $M_p$, as $n\to\infty$.
Extending, if necessary, the probability space on which $(\widehat{N}_n)_{n\in\N}$ and $\widehat{N}^{(c,1)}$ are defined,  we can independently construct GW processes $(\widehat{X}_{i,k})_{i\in \N, k\in\N_0}$ having the same law as $({X}_{i,k})_{i\in \N, k\in\N_0}$. Write
$$
\widehat{Z}_n(t):=\sum_{k=0}^{[nt]} \sum_{i=1}^{\exp(n y_k^{(n)})} \widehat{X}_{i,k}([nt]-k),\quad n\in\mn_0,\quad  t\geq 0,
$$
so that for each $n\in\N$, the distributions of the processes $(\widehat Z_n(t))_{t\geq 0}$ and $(Y_{[nt]})_{t\geq 0}$ coincide. Fix some $T>0$ and let $d_T$ be the standard $J_1$-metric on the Skorokhod space $D[0,T]$. Our aim is to prove that with probability $1$,
\begin{equation}\label{eq:need}
\lim_{n\to\infty} d_T\left(\frac 1n \log^+ \widehat{Z}_n(\cdot),  \underset{\tau_k\leq
\cdot}{\sup}\big(y_k + (\cdot-\tau_k)\log\mu\big)\right) =0.
\end{equation}

\vspace*{2mm}
\noindent
\textsc{Step 2: Estimate for non-extremal order statistics.}
We shall decompose the process $(\widehat{Z}_n(t))_{t\geq 0}$ into the contribution coming from times with extremely active immigration, and the contribution of all the other times.
For a truncation parameter $0< \gamma < 1$, put
\begin{align*}
\widehat Z_{n}^{(\leq \gamma)}(t)&:=\sum_{k=0}^{[nt]}\1_{\{y_k^{(n)}\leq \gamma\}}\sum_{i=1}^{\exp(ny_k^{(n)})}\widehat X_{i,k}([nt]-k),\\
\widehat Z_{n}^{(>\gamma)}(t) &:=\sum_{k=0}^{[nt]}\1_{\{y_k^{(n)}>\gamma\}}\sum_{i=1}^{\exp(ny_k^{(n)})}\widehat X_{i,k}([nt]-k),
\end{align*}
so that
\begin{equation}\label{equ}
\widehat Z_{n}(t) = \widehat Z_{n}^{(\leq \gamma)}(t) + \widehat Z_{n}^{(>\gamma)}(t),\quad n\in\N_0, \; t\geq 0.
\end{equation}
Suppose for a moment that with probability $1$,
\begin{equation}\label{eq:need1}
\underset{\gamma\to 0+}{\lim}\,\underset{n\to\infty}{\lim\sup}\,d_T\bigg(\frac 1n\log^+ \widehat Z_n^{(>\gamma)}(\cdot),\,\underset{\tau_k\leq\cdot}{\sup}\,\big(y_k + (\cdot-\tau_k)\log\mu\big)\bigg)=0,
\end{equation}
where $\gamma$ is restricted to the set $\{1/m\colon m=2,3,\ldots\}$. Let us argue that~\eqref{eq:need1} implies~\eqref{eq:need}.

\vspace*{2mm}
\noindent
\textit{Proof in the case $\mu\leq 1$.} Using \eqref{equ} in combination with Lemma~\ref{lem:log_+_subadd} yields
\begin{equation*}
\log^+(\widehat Z_{n}^{(>\gamma)}(t))
\leq
\log^+(\widehat Z_{n}(t))
\leq
\log^+(\widehat Z_{n}^{(>\gamma)}(t)) + \log^+(\widehat Z_{n}^{(\leq \gamma)}(t))+2\log 2.
\end{equation*}
Since the processes $(\widehat Z_n^{(\leq \gamma)}(t))_{t\geq 0}$ and $(Y_{[nt]}^{(\leq \gamma)})_{t\geq 0}$ (c.f.~\eqref{eq:def_Y_proc_truncated}) have the same distribution, we can use Lemma~\ref{aux3} and the Borel-Cantelli lemma to conclude that with probability $1$,
$$
\underset{n\to\infty}{\lim\sup}\,\underset{0\leq t\leq T}{\sup}\, \frac 1n \log^+(\widehat Z_n^{(\leq \gamma)}(t)) \leq \gamma.
$$
It follows from the last two inequalities that with probability $1$,
$$
 \underset{n\to\infty}{\lim\sup}\,\underset{0\leq t\leq T}{\sup}\, \left|\frac 1n \log^+(\widehat Z_{n}(t)) - \frac 1n \log^+(\widehat Z_{n}^{(>\gamma)}(t))\right| \leq \gamma.
$$
The Skorokhod distance is majorized by the sup-distance, whence we conclude that with probability $1$,
$$
\underset{n\to\infty}{\lim\sup}\, d_T\left(\frac 1n \log^+ \widehat Z_{n}(\cdot), \frac 1n \log^+ \widehat Z_{n}^{(>\gamma)}(\cdot)\right) \leq \gamma.
$$
The triangle inequality entails  that~\eqref{eq:need1} implies~\eqref{eq:need}.

\vspace*{2mm}
\noindent
\textit{Proof in the case $\mu > 1$.}
Our aim is to obtain an upper bound for $\log^+(\widehat Z_{n}(t))$. Since the processes $(\widehat Z_n^{(\leq \gamma)}(t))_{t\geq 0}$ and $(Y_{[nt]}^{(\leq \gamma)})_{t\geq 0}$ have the same distribution,  Lemma~\ref{aux3} in conjunction with the Borel-Cantelli lemma allows us to conclude that with probability $1$,
$$
\underset{n\to\infty}{\lim\sup}\, \underset{0\leq t\leq T}{\sup}\, \frac 1n \log^+ \left(\mu^{-[nt]} \widehat  Z_n^{(\leq \gamma)}(t)\right) \leq \gamma.
$$
Since $\log x\leq \log^+ x$ and $[nt]\leq nt$, it follows that with probability $1$, for sufficiently large $n$,
\begin{equation}\label{eq:est1}
\widehat  Z_n^{(\leq \gamma)}(t) \leq e^{n(t\log \mu+ 2\gamma)}, \quad 0\leq t \leq T.
\end{equation}
Noting that $e^{n(t\log \mu+ 2\gamma)}\geq e^{2n\gamma}>1$ for sufficiently large $n$ and all $t\geq 0$, we obtain the estimate
\begin{align*}
%\log^+(\widehat Z_{n}^{(>\gamma)}(t))
%\leq
\log^+(\widehat Z_{n}(t))
&\leq
\log^+\left(\widehat Z_{n}^{(>\gamma)}(t)  + e^{n(t\log \mu + 2\gamma)}\right)\\
&\leq
\log\left(\widehat Z_{n}^{(>\gamma)}(t)  + e^{n(t\log \mu + 2\gamma)}\right)\\
&\leq
n(t\log \mu - 2\sqrt \gamma) + \log\left( e^{-n(t\log \mu - 2\sqrt \gamma)} \widehat Z_{n}^{(>\gamma)}(t) + e^{4n\sqrt \gamma}\right).
\end{align*}
Using again the inequality $\log x\leq \log^+ x$ and then Lemma~\ref{lem:log_+_subadd}, we arrive at
\begin{equation}\label{eq:222}
\log^+(\widehat Z_{n}(t))
\leq
n(t\log \mu - 2\sqrt \gamma)  + \log^+ \left(e^{-n(t\log \mu - 2\sqrt \gamma)} \widehat Z_{n}^{(>\gamma)}(t)\right) + 4n\sqrt \gamma +2\log 2.
\end{equation}
Below we shall prove that with probability $1$ there exist a random $0< \gamma_0 < 1$ and $n_1\in\N$ such that for all $0<\gamma<\gamma_0$ and $n>n_1$, we have
\begin{equation}\label{eq:333}
\widehat Z_{n}^{(>\gamma)}(t) \geq  e^{n(t\log \mu - 2\sqrt \gamma)}, \quad \sqrt \gamma\leq t\leq  T.
\end{equation}
Given~\eqref{eq:333}, we conclude that $\log^+$ on the right-hand side of~\eqref{eq:222} can be replaced by $\log$, thus yielding for all $\sqrt \gamma\leq t\leq  T$ the estimate
\begin{equation*}
\log^+(\widehat Z_{n}^{(>\gamma)}(t)) \leq \log^+(\widehat Z_{n}(t))
\leq
4n\sqrt \gamma + \log^+ (\widehat Z_{n}^{(>\gamma)}(t)) + 2\log 2,
\end{equation*}
where the first inequality is an immediate consequence of~\eqref{equ}. For $0\leq t \leq \sqrt \gamma$, Lemma~\ref{lem:log_+_subadd} and~\eqref{eq:est1} yield
$$
\log^+(\widehat Z_{n}^{(>\gamma)}(t)) \leq \log^+(\widehat Z_{n}(t))
\leq
n(\sqrt{\gamma} \log \mu + 2\gamma) + \log^+ (\widehat Z_{n}^{(>\gamma)}(t)) + 2\log 2.$$
From now on, we can argue as in the case $\mu\leq 1$ to conclude that~\eqref{eq:need1} implies~\eqref{eq:need}

\vspace*{2mm}
\noindent
\textit{Proof of~\eqref{eq:333}}.  First we prove that with probability $1$ there is a $\gamma_0>0$ such that for  all  $0<\gamma<\gamma_0$ at least one atom $(\tau, y)$ of the point process $\widehat N^{(c,1)}$ satisfies $0<\tau <\sqrt \gamma$ and $y>\gamma$. Indeed, the number of points of $\widehat N^{(c,1)}$ in the set $(0,1/\sqrt{2k})\times (1/k,\infty)$ is Poisson-distributed with parameter $c\sqrt{k/2}$, for all $k\in\N$. Since $\sum_{k\geq 1} e^{-c\sqrt{k/2}}$ is finite,  the Borel-Cantelli lemma implies that
with probability $1$ we can find $k_0\in\N$ such that for every $k\geq k_0$ at least one atom $(\tau, y)$ of $\widehat N^{(c,1)}$ satisfies $0<\tau < 1/\sqrt{2k}$ and $y>1/k$. If $1/(k+1)\leq \gamma \leq  1/k$, then it follows $0<\tau <\sqrt \gamma$ and $y>\gamma$, so that we can take $\gamma_0 = 1/k_0$.

Since with probability $1$, $\widehat N_n$ converges to $\widehat N^{(c,1)}$ vaguely on $M_p$, there is a random $n_0\in\N$ such that for all $n\geq n_0$ at least one atom, say $(k_n/n, z_n)$, of $\widehat N_n$ satisfies $k_n/n<\sqrt \gamma$ and $z_n>\gamma$. In the following, we condition on the $\sigma$-field generated by $\widehat N^{(c,1)}$ and $\widehat N_n$, so that we can view $\gamma_0$ and $n_0$ as deterministic quantities.

Recall that we consider the case $\mu>1$. As has already been mentioned in the proof of Lemma \ref{aux2} (case $\mu>1$), %  By a theorem of  Seneta-Heyde (see Theorem 5.1 on p.~83 and Corollary 5.3 on p.~85 in \cite{Asmussen+Hering:1983}),
there exists a function $L$ slowly varying at $\infty$ such that
$$
\frac{X_{1,0} (m)}{\mu^m L(\mu^m)} \to W \quad \text{a.s. as } m\to\infty,
$$
the limit random variable $W$ being a.s.\ positive on the survival event of the GW process $X_{1,0}$. It follows that there is $\varepsilon>0$ such that
$$
\Prob\left\{\inf_{m\in\N}\frac{\widehat X_{i,k} (m)}{\mu^m L(\mu^m)} > \varepsilon \right\} >\varepsilon, \quad i\in \N,~ k\in\N_0.
$$
Since $z_n>\gamma$ for $n\geq n_0$ we have
$$
\sum_{n\geq n_0} \Prob\left\{\inf_{m\in \N}\frac{\widehat X_{i,k_n} (m)}{\mu^m L(\mu^m)} \leq \varepsilon \text{ for all } i=1,\ldots, e^{nz_n}\right\}
<
\sum_{n\geq n_0} (1-\varepsilon)^{e^{nz_n}} < \infty.
$$
By the Borel-Cantelli lemma, for sufficiently large $n$, there is an immigrant $i_n$ arriving at time $k_n< n\sqrt \gamma$ whose offspring numbers satisfy $\widehat X_{i_n,k_n}(m)>\varepsilon \mu^m L(\mu^m)$ for all $m\in\N$. For all $\sqrt \gamma \leq t \leq T$ and sufficiently large $n\geq n_1$ we have
$$
\widehat Z_n^{(>\gamma)} (t) \geq \widehat X_{i_n,k_n}([nt] - k_n) >
\varepsilon \mu^{[nt] - k_n} L(\mu^{[nt] - k_n})
>
e^{n(t\log \mu - 2 \sqrt \gamma)},
$$
thereby completing the proof of~\eqref{eq:333}.

\vspace*{2mm}
\noindent
\textsc{Step 3: Enumerating the points.} In the following we prove~\eqref{eq:need1}.
Let $\mathcal{F}$ be the $\sigma$-field generated by $(y_k^{(n)})_{k\in\N_0, n\in\N}$ and $(\tau_k,\,
y_k)_{k\in\N_0}$. Until further notice we work conditionally on $\mathcal{F}$, so that all $\mathcal{F}$-measurable variables can be treated as deterministic constants. After discarding an event of probability $0$, we can assume that the points $(\tau_k,
y_k)$, $k\in\N_0$, have the following properties:

\vspace*{1mm}
\noindent {\sc Property 1}: $\tau_k \notin \{0, T\}$ and $y_k\notin \{1/2,1/3,\ldots,\infty\}$ for all $k\in\N_0$.

\vspace*{1mm}
\noindent {\sc Property 2}: $\sup_{\tau_k\leq t}\big(y_k+(t-\tau_k)\log\mu\big)\geq 0$ for all $t\geq 0$.

\vspace*{1mm}
\noindent {\sc Property 3}: $\tau_k\neq\tau_j$ a.s.\ for $k\neq j$.

\vspace*{1mm}
\noindent {\sc Property 4}: $(0,0)$ is an accumulation point of $(\tau_k,
y_k)$, $k\in\N_0$.

Relation \eqref{7} implies that  for large
enough $n$ and some $p\in \N$,
$$
\widehat{N}_n ([0,T]\times
(\gamma,\infty])=\widehat{N}^{(c,1)}([0,T]\times (\gamma,\infty])=p,
$$
where $p\neq 0$ if $\gamma$ is sufficiently small. Denote by
$(\bar{\tau}_j,\bar{y}_j)_{1\leq j\leq p}$ an enumeration of the
points of $\widehat{N}^{(c,1)}$ in $[0,T]\times (\gamma,\infty]$ with
$0<\bar{\tau}_1<\bar{\tau}_2<\ldots< \bar{\tau}_p<T$
and by $(\bar{\tau}_j^{(n)}, \bar{y}_j^{(n)})_{1\leq j\leq p}$
the analogous enumeration of the points of $\widehat{N}_n$ in $[0,T]\times (\gamma,\infty]$. Then, relation~\eqref{7} implies that
(possibly, after renumbering the points),
\begin{equation}\label{2.2}
\lim_{n\to\infty} \bar{\tau}^{(n)}_j = \bar{\tau}_j,
\quad
\lim_{n\to\infty} \bar{y}^{(n)}_j =\bar{y}_j,
\quad
j=1,\ldots,p.
\end{equation}
For sufficiently large $n\in\N$, $j=1,\ldots, p$, and $t\geq 0$, set
$$
U_{n,j}(t):=\sum_{i=1}^{\exp(n\bar{y}_j^{(n)})}\widehat X_{i,\,n\bar{\tau}_j^{(n)}}([nt]),
$$
so that $\widehat Z_n^{(>\gamma)}(t)=\sum_{\bar{\tau_j}^{(n)}\leq t}U_{n,j}(t-\bar{\tau}_j^{(n)})$. Put also
$$
Z_{n,j}(t):=n^{-1}\log^+U_{n,j}(t),
\quad
Z_j(t):=(\bar{y}_j+t \log\mu )^+.
$$
For later needs, we also define these functions to be zero for $t<0$.
We rewrite~\eqref{eq:need1} in the following form:
\begin{equation}\label{princ10}
\underset{\gamma\to 0+}{\lim}\,\underset{n\to\infty}{\lim\sup}\,d_T\bigg(\frac 1n \log^+\bigg(\sum_{\bar{\tau_j}^{(n)}\leq \cdot}U_{n,j}(\cdot-\bar{\tau}_j^{(n)})\bigg),\,\underset{\tau_k\leq\cdot}{\sup}\,\big(y_k+(\cdot-\tau_k)\log\mu \big)\bigg)=0.
\end{equation}

\vspace*{2mm}
\noindent
\textsc{Step 4: Proof of~\eqref{princ10}}.
By the triangle inequality, we have
\begin{multline}\label{tech1}
d_T\bigg(\frac 1n \log^+\bigg(\sum_{\bar{\tau}_j^{(n)}\leq \cdot}U_{n,j}(\cdot-\bar{\tau}_j^{(n)})\bigg),\underset{\tau_k\leq\cdot}{\sup}\,\big(y_k+(\cdot-\tau_k)\log\mu\big)\bigg) %
\\
\begin{aligned}
&\leq d_T\bigg(\frac 1n \log^+\bigg(\sum_{\bar{\tau}_j^{(n)}\leq \cdot}U_{n,j}(\cdot-\bar{\tau}_j^{(n)})\bigg), \underset{\bar{\tau}_j\leq \cdot}{\sup}\,Z_j(\cdot-\bar{\tau}_j)\bigg)
\\&+
\underset{0\leq t\leq T}{\sup}\,\big|\underset{\tau_k\leq t}{\sup}\,\big(y_k+(t-\tau_k)\log\mu\big)-\underset{\bar{\tau}_j\leq t}{\sup}\,Z_j(t-\bar{\tau}_j)\big|,
\end{aligned}
\end{multline}
where for the last term we have used the fact that the Skorokhod metric $d_T$ is dominated by the uniform metric on $[0,T]$. In the following, we estimate both terms on the right-hand side.

\vspace*{2mm}
\noindent
\textit{First term in~\eqref{tech1}.}
We intend to check that
\begin{equation}\label{inter}
\lin d_T\bigg(\frac 1n\log^+\bigg(\sum_{\bar{\tau}_j^{(n)}\leq \cdot}U_{n,j}(\cdot-\bar{\tau}_j^{(n)})\bigg),\, \underset{\bar{\tau}_j\leq \cdot}{\sup}\,Z_j(\cdot-\bar{\tau}_j)\bigg)=0 \quad \text{a.s.}
\end{equation}
In view of
\begin{eqnarray*}
\underset{\bar{\tau}_j^{(n)}\leq t}{\sup}\,\log^+\big(U_{n,j}(t-\bar{\tau}_j^{(n)})\big)&\leq& \log^+\bigg(\sum_{\bar{\tau}_j^{(n)}\leq t}U_{n,j}(t-\bar{\tau}_j^{(n)})\bigg)\\&\leq& \log^+\#\{j: \bar{\tau}_j^{(n)}\leq t\}+\underset{\bar{\tau}_j^{(n)}\leq t}{\sup}\,\log^+\big(U_{n,j}(t-\bar{\tau}_j^{(n)})\big)\\&\leq& \log p+\underset{\bar{\tau}_j^{(n)}\leq t}{\sup}\,\log^+\big(U_{n,j}(t-\bar{\tau}_j^{(n)})\big) \end{eqnarray*}
for $t\in [0,T]$, Equation~\eqref{inter} is equivalent to
\begin{equation}\label{inter2}
\lin d_T\bigg(\underset{\bar{\tau}_j^{(n)}\leq\cdot}{\sup}\,Z_{n,j}(\cdot-\bar{\tau}_j^{(n)}),\, \underset{\bar{\tau}_j\leq \cdot}{\sup}\,Z_j(\cdot-\bar{\tau}_j)\bigg)=0 \quad \text{a.s.}
\end{equation}
In view of \eqref{2.2}, an application of Lemma \ref{aux2} yields
\begin{equation}\label{imp2}
Z_{n,j}(t)\quad \to\quad Z_j(t),\quad j=1,\ldots, p
\end{equation}
a.s.\ uniformly on $[0,T]$ (recall that we work conditionally on $\mathcal{F}$).
Define $\lambda_n$ to be continuous and strictly increasing
functions on $[0,T]$ with $\lambda_n(0) =0$, $\lambda_n(T) =T$,
$\lambda_n(\bar{\tau}_j)=\bar{\tau}^{(n)}_j$ for
$j=1,\ldots,p$, and let $\lambda_n$ be linearly interpolated
elsewhere on $[0,T]$. It is easily seen that $\lin \underset{0\leq t\leq T}{\sup}\,|\lambda_n(t)-t|=0$.
This implies that
\begin{equation}\label{imp1}
\lin \underset{t\in[0,T]}{\sup}\,\big|Z_{n,j}\big(\lambda_n(t)-\bar{\tau}_j^{(n)}\big)-Z_j(t-\bar{\tau}_j)\big|=0,\quad j=1,\ldots,p
\end{equation}
a.s. Indeed, for $t\in [0,\bar{\tau}_j)$ we have $Z_{n,j}(\lambda_n(t)-\bar{\tau}_j^{(n)})=Z_j(t-\bar{\tau}_j)=0$.
Also, as a consequence of \eqref{2.2} and \eqref{imp2} we obtain the relation
$$
\lin\underset{t\in [0,T-\bar{\tau}_j]}{\sup}\,|Z_{n,j}(\lambda_n(t+\bar{\tau}_j)-\bar{\tau}_j^{(n)})-Z_j(t)|=0,
$$
which proves \eqref{imp1}.
Now \eqref{inter2} follows from
\begin{eqnarray*}
&&\underset{t\in[0,T]}{\sup}\,\big|\underset{\bar{\tau}_j^{(n)}\leq \lambda_n(t)}{\sup}\,Z_{n,j}\big(\lambda_n(t)-\bar{\tau}_j^{(n)}\big)-\underset{\bar{\tau}_j\leq t}{\sup}\,Z_j(t-\bar{\tau}_j)\big|\\&=&
\underset{t\in[0,T]}{\sup}\,\big|\underset{\bar{\tau}_j\leq t}{\sup}\,Z_{n,j}\big(\lambda_n(t)-\bar{\tau}_j^{(n)}\big)-\underset{\bar{\tau}_j\leq t}{\sup}\,Z_j(t-\bar{\tau}_j)\big|\\&\leq&
\underset{t\in[0,T]}{\sup}\,\sum_{\bar{\tau}_j\leq t}\big|Z_{n,j}\big(\lambda_n(t)-\bar{\tau}_j^{(n)}\big)-Z_j(t-\bar{\tau}_j)\big|\\&\leq& \sum_{j=1}^p \underset{t\in[0,T]}{\sup}\,\big|Z_{n,j}\big(\lambda_n(t)-\bar{\tau}_j^{(n)}\big)-Z_j(t-\bar{\tau}_j)\big|
\end{eqnarray*}
because the right-hand side converges to zero a.s.\ by \eqref{imp1}.

\vspace*{2mm}
\noindent
\textit{Second term in~\eqref{tech1}.}
Left with proving that
\begin{equation}\label{imp3}
\underset{\gamma\to 0+}{\lim}\,\underset{0\leq t\leq T}{\sup}\,\big|\underset{\tau_k\leq t}{\sup}\,\big(y_k+(t-\tau_k)\log\mu\big)-\underset{\bar{\tau}_j\leq t}{\sup}\,\big(\bar{y}_j+(t-\bar{\tau}_j)\log\mu\big)^+\big|=0,
\end{equation}
we first recall that the points $(\bar{\tau}_1, \bar{y_1}), \ldots, (\bar{\tau}_p, \bar{y}_p)$ belong to the collection $(\tau_k,y_k)_{k\in\mn_0}$. Hence, for $t\in [0,T]$,
\begin{eqnarray}\label{3}
\sup_{\tau_k\leq t}\big(y_k+(t-\tau_k)\log\mu\big)
\geq
\sup_{\bar{\tau}_j\leq t}\big(\bar{y}_j+(t-\bar{\tau}_j)\log\mu\big)^+,
\end{eqnarray}
where we also used that the supremum on the left-hand side is nonnegative by Property~2.
Pick now
$\tau_k\notin\{\bar{\tau}_1,\ldots, \bar{\tau}_p\}$ satisfying $\tau_k\leq t$.
Recall that all $y_k$ other than
$\bar{y}_1,\ldots, \bar{y}_p$ do not exceed $\gamma$.
If $\mu\leq 1$, we infer
\begin{equation*}
y_k+(t-\tau_k)\log\mu
\leq y_k
\leq \gamma
\leq \gamma + \sup_{\bar{\tau}_j\leq t}(\bar{y}_j+(t-\bar{\tau}_j)\log\mu)^+.
\end{equation*}
Together with \eqref{3} this proves~\eqref{imp3}.  In the following, let $\mu>1$. Fix some $\delta>0$. It suffices to show that for sufficiently small $\gamma>0$ we have
\begin{equation}\label{3_suff}
y_k+(t-\tau_k)\log\mu \leq \sup_{\bar{\tau}_j\leq t}\big(\bar{y}_j+(t-\bar{\tau}_j)\log\mu\big)^+ + \gamma+\delta
\end{equation}
for all $k\in\N_0$ such that $\tau_k \leq t$. If $y_k > \gamma$, then  $(\tau_k,y_k)$ is one of the points $(\bar{\tau}_1, \bar{y}_1),\ldots, (\bar{\tau}_p, \bar{y}_p)$, and~\eqref{3_suff} is evident. Let therefore $y_k\leq \gamma$. Then,
$$
y_k + (t-\tau_k) \log \mu \leq \gamma + t\log \mu.
$$
This immediately implies~\eqref{3_suff} if $t\leq \delta/\log \mu$. Therefore, let $t>\delta/\log \mu$.  If $\gamma>0$ is sufficiently small, then by Property~4 we can find a point $(\bar{\tau}_{j'}, \bar{y}_{j'})$ such that $\bar{\tau}_{j'} < \delta/\log \mu$. We infer
$$
\sup_{\bar{\tau}_j\leq t}\big(\bar{y}_j+(t-\bar{\tau}_j)\log\mu\big)^+
\geq
\bar{y}_{j'}+(t-\bar{\tau}_{j'})\log\mu
\geq t\log \mu -\delta.
$$
Taking the last two inequalities  together we arrive at~\eqref{3_suff}, which proves \eqref{3}. The proof of Theorem \ref{main1} is complete.
\end{proof}

The proof of Theorem \ref{main2} runs the same path as that of Theorem \ref{main1}. We note that the proof is essentially based on the convergence
\begin{equation*}
\sum_{k\geq 0}\1_{\{J_{k}\neq 0\}}\varepsilon_{\left(\frac kn,\,
\frac {\log J_{k}}{b_n}\right)}\quad \Rightarrow\quad
N^{(1,\alpha)},\quad n\to\infty
\end{equation*}
on $M_p$ and uses the corresponding part of Lemma \ref{aux3} together with Lemma \ref{aux2a} in which we take $c_n=b_n$. We refrain from discussing the details which are much simpler here. 

\vspace{1cm}
\noindent   {\bf Acknowledgements}  \quad
\footnotesize
A part of this work was done while A.~Iksanov was visiting M\"{u}nster in
July 2015, 2016. He gratefully acknowledges hospitality and the financial support by DFG SFB 878 ``Geometry, Groups and Actions''.

\end{document}